\documentclass[reqno,12pt]{amsart}

\usepackage{amsfonts,amsthm,amsmath,amssymb}
\usepackage{hyperref}
\usepackage{enumerate, bbm}
\hypersetup{colorlinks=false,pdfborder={0 0 0}}
\usepackage[margin=1.5in]{geometry}

\numberwithin{equation}{section}

\newtheorem{theorem}{Theorem}[section]
\newtheorem{lemma}[theorem]{Lemma}
\newtheorem{proposition}[theorem]{Proposition}
\newtheorem{remark}[theorem]{Remark}

\renewcommand{\leq}{\leqslant}

\renewcommand{\geq}{\geqslant}

\renewcommand{\d}{\mathrm{d}}
\newcommand{\ve}{\varepsilon}
\newcommand{\ZZ}{\mathbb{Z}}
\newcommand{\RR}{\mathbb{R}}
\newcommand{\CC}{\mathbb{C}}
\newcommand{\NN}{\mathbb{N}}
\renewcommand{\AA}{\mathbb{A}}

\newcommand{\x}{\mathbf{x}}
\newcommand{\y}{\mathbf{y}}
\newcommand{\z}{\mathbf{z}}
\newcommand{\w}{\mathbf{w}}
\newcommand{\h}{\mathbf{h}}
\newcommand{\kk}{\mathbf{k}}

\DeclareMathOperator{\supp}{supp}

\begin{document}

\title[Prime points on smooth hypersurfaces]{Prime points on smooth hypersurfaces}

\author{Yijie Diao}
\address{ISTA\\
Am Campus 1\\
3400 Klosterneuburg\\
Austria}
\email{yijie.diao@ist.ac.at}

\author{Shuntaro Yamagishi}
\address{Mathematisches Institut\\
Universit\"at Bonn\\
Endenicher Allee 60\\
53115 Bonn\\
Germany}
\email{yamagish@math.uni-bonn.de}

\date{\today}

\begin{abstract}
Let $F \in \mathbb{Z}[x_1, \ldots, x_n]$ be a homogeneous form of degree $d \geq 4$ which defines a smooth
hypersurface in $\mathbb{P}^{n-1}_{\CC}$. For $n \geq 24 d^4 2^d$,
we prove an asymptotic formula for the number of prime solutions to the equation $F(x_1, \ldots, x_n) = 0$,
provided $F$ satisfies suitable local conditions.
\end{abstract}

\maketitle
\thispagestyle{empty}
\setcounter{tocdepth}{1}

\tableofcontents

\section{Introduction}\label{s:introduction}
Let $F \in \ZZ[x_1,\ldots,x_n]$ be a homogeneous form of degree $d \geq 2$.
We are interested in solutions to $F(\x) = 0$ whose coordinates are all primes or prime powers.
For $d \geq 3$, Cook and Magyar~\cite{cook} first obtained a general theorem regarding the asymptotic formula
for the quantity
$$
N_F(P) = \sum_{\x \in [1, P]^n \cap \ZZ^n} \Lambda(x_1) \cdots \Lambda(x_n) \mathbbm{1}_{V(F)}(\x),
$$
which is the number of integer solutions, weighted by the von Mangoldt function, to $F(\x) = 0$ in $[1, P]^n$.
For quadratic forms, the corresponding result was known due to Liu~\cite{liu} and Zhao~\cite{zhao}, while more recently Green established an asymptotic formula in eight prime variables under a mild nondegeneracy condition~\cite{green}. The notable achievement of Cook and Magyar required $n - \dim V_F^*$ to be an exponential tower in $d$, where $V_F^*$ denotes the singular locus of the affine variety
$V(F) = \{ \mathbf{z} \in {\mathbb{A}}^n_{\CC}: F(\mathbf{z}) = 0 \}$.
The second author~\cite{yamagishi} improved this, with smooth weights, to be
$$
n - \dim V_F^* \geq 2^8 3^4 5^2 d^3 (2d-1)^2 4^{d}.
$$
Liu and Zhao~\cite{liuzhao} developed a different method based on mean value estimates, leading to the following further improvement
$$
n - \dim V_F^* \geq 12d(d-1)4^d+7(d-1)2^d+4d+8.
$$
The required codimension of $V_F^*$ of order $4^d$, up to a polynomial in $d$, is a substantial improvement compared to the original exponential tower in $d$ in Cook and Magyar's work. However, this remains far from Birch's result~\cite{birch} for integer solutions, which requires $n - \dim V_F^* > (d-1) 2^d$.
In this paper, we show that for $F$ defining a smooth hypersurface in $\mathbb{P}_{\CC}^{n-1}$, we can reduce the number of variables needed  to  order $d^4 2^d$.
\begin{theorem}
\label{t:main}
Let $F \in \ZZ[x_1, \ldots, x_n]$ be a homogeneous form of degree $d \geq 4$ with $V_F^{*} = \{ \mathbf{0} \}$.
Suppose
$$
n \geq 24  d^4  2^{d}.
$$
Suppose further that the equation $F(\x) = 0$ has a nonsingular real solution in $(0,1)^{n}$, and also
has a nonsingular solution in $( \mathbb{Z}_p^{\times} )^n$ for every prime $p$.
Then
$$
N_F(P) \gg P^{n - d}.
$$
\end{theorem}
We in fact obtain an asymptotic formula for a smooth variant of $N_F(P)$ whenever
$$
n \geq (24 (d-1)^2 d^2 + 7 (d-1) ) 2^{d} + 11,
$$
from which Theorem \ref{t:main} follows directly. Although we prove the theorem in terms of the von Mangoldt function $\Lambda$, it can be replaced with
$\Lambda^*$, defined by $\Lambda^*(x) = \log x$ if $x$ is  prime and $0$ otherwise, without much additional effort.
Also, the present paper does not address cubic forms; these will be treated separately in a forthcoming paper.
Our proof relies on the mean value estimate developed by Liu and Zhao, and
a simple structure of $F$ defining a smooth hypersurface. Roughly speaking, in \cite[Section~7]{liuzhao}, the factor $4^d$ arises from the $2^d$ needed to apply Birch's work over the integers, together with $2^d$ coming from estimating Weyl sums over primes on  the minor arcs, both of which rely on Weyl differencing.
However, for certain polynomials, the latter estimate can be obtained without Weyl differencing by incorporating Vinogradov's mean value theorem.
Since $F$ defining a smooth hypersurface always contain a monomial of the form $x_i^d$ or $x_j^{d-1} x_k$,
this allows us to restrict our attention to two types of Weyl sums over primes, both of which can be treated using Vinogradov's mean value theorem.
Our argument does require the major arcs, defined in Section~\ref{s:prelim}, to be extended from $\mathfrak{M}((\log P)^A)$ to $\mathfrak{M}(P^\vartheta)$
for appropriately chosen $A, \vartheta > 0$, but this has already been done in the second author's work~\cite{yamagishi, yamagishi2}.

For the corresponding problem for systems of polynomials involving differing degrees, the second author~\cite{yamagishi3} generalised the argument of Cook and Magyar,
which required the number of variables to be an exponential tower in the degrees. More recently, Liu and Xie~\cite{liuxie} generalised the
method in~\cite{liuzhao} to improve this  to be an exponential in the degrees. It could be interesting to investigate to which systems of polynomials, if any,
the idea of this paper could be generalised to.  Finally, $\ve$ denotes an arbitrarily small positive number,
and its value may vary across different occurrences.


\subsection*{Acknowledgements} The first author thanks his advisor Tim Browning for his support,
and Christian Bernert for introducing the paper by Liu and Xie~\cite{liuxie} to him. 
He also thanks the Hausdorff Center for Mathematics for its hospitality during his visit.
The second author wishes to thank the Mathematical Institute at the University of Bonn for providing an excellent working environment.
The authors thank Kien Huu Nguyen for helpful clarifications concerning~\cite{nguyen1}.

\subsection*{AI usage disclosure} The central technical result of the paper is Theorem \ref{t:minor}, which gives the minor arcs estimate.
GPT-5.6 Sol worked out a first version of this estimate in response to prompts proposing to strengthen the minor arcs estimate by using specifically the structure of nonsingular forms described in Lemma~\ref{smooth}.


\section{Preliminaries}
\label{s:prelim}
Let $G \in \ZZ[x_1, \ldots, x_n]$ be a homogeneous form of degree $d \geq 2$.
We let $V^*_{G}$ denote the singular locus of $V(G) = \{  \mathbf{z} \in \mathbb{A}^n_{\CC}:  G(\mathbf{z}) = 0  \}$, i.e. it is the affine variety
\begin{equation}
\label{sing loc}
V_{G}^* =  \{ \mathbf{z} \in \mathbb{A}_{ \mathbb{C} }^{n}:  \nabla G(\mathbf{z}) = \mathbf{0} \},
\end{equation}
where
$\nabla G = \bigl( \frac{\partial G}{\partial x_1}, \ldots,  \frac{\partial G}{\partial x_n} \bigr)$.

Let $F \in \ZZ[x_1, \ldots, x_n]$ be a homogeneous form of degree $d \geq 2$.
Let $\x_0 \in (0, 1)^n$ be a nonsingular solution to $F(\x) = 0$, and let $\omega: \RR \to [0, 1]$ be a smooth function such that
$\omega(0) > 0$ and $\supp \omega \subseteq [- \delta, \delta]$ for $\delta > 0$ sufficiently small with respect to $F$ and $\x_0$.
Let $P \geq 1$. We define
\begin{eqnarray}
\label{defS}
S(\alpha) = \sum_{\x \in \ZZ^n}   \prod_{i = 1}^n  \psi_i (x_i)  \Lambda(x_i)  \cdot e (\alpha F(\x) ),
\end{eqnarray}
where
\begin{eqnarray}
\label{defn of psi}
\psi_i(x_i) = \omega \left( \frac{x_i}{P} - x_{0,i}  \right) \qquad (1 \leq i \leq n).
\end{eqnarray}

Given $L \geq 1$, let
$$
\mathfrak{M}(L)
=
\bigcup_{1 \leq q \leq L} \bigcup_{\substack{ 0 \leq  a < q  \\ \gcd(a,q) = 1 }} \left\{   \alpha \in \RR /\ZZ: \left\| \alpha - \frac{a}{q}  \right\| \leq \frac{L}{q P^d}  \right\},
$$
where $\| \beta \| = \min_{m \in \ZZ} | \beta - m |$. Let
\begin{eqnarray}
\label{defnL}
\mathfrak{n}(L) = \mathfrak{M}(2L) \setminus \mathfrak{M}(L).
\end{eqnarray}
Then
\begin{eqnarray}
\label{measnL}
|\mathfrak{n}(L)| \ll L^2 P^{-d}.
\end{eqnarray}
Dirichlet's theorem gives $\mathfrak{M}(P^{d/2}) = \RR/ \ZZ$.
Set $Q_0 = P^{1/100}$, and let $J$ be the least integer for which
$2^J Q_0 \geq P^{d/2}$. Then it follows that
\begin{eqnarray}
\label{dydshell}
(\RR / \ZZ) \setminus \mathfrak{M}(Q_0) = \bigcup_{ 1 \leq j \leq J } \mathfrak{n}( 2^{j-1} Q_0 ),
\end{eqnarray}
since
$$
(\RR / \ZZ) \setminus \mathfrak{M}( 2^{J-1} Q_0 ) =  \mathfrak{M}( 2^{J} Q_0 ) \setminus \mathfrak{M}( 2^{J-1} Q_0 ) = \mathfrak{n}(2^{J-1} Q_0).
$$

We prove the following asymptotic formula for a smooth variant of $N_F(P)$.
\begin{theorem}
\label{t:asymptotic}
Let $F \in \ZZ[x_1, \ldots, x_n]$ be a homogeneous form of degree $d \geq 4$ with $V_F^{*} = \{ \mathbf{0} \}$.
Suppose
$$
n \geq (24 (d-1)^2 d^2 + 7 (d-1) ) 2^{d} + 11.
$$
Suppose further that the equation $F(\x) = 0$ has a nonsingular real solution in $(0,1)^{n}$, and also
has a nonsingular solution in $( \mathbb{Z}_p^{\times} )^n$ for every prime $p$.
Let $\x_0$, $\delta$, $\omega$ and $\psi_1,\ldots,\psi_n$ be as in this section.
Then, for any $A > 0$, we have
\[
\sum_{\x\in\ZZ^n}\prod_{i=1}^n\psi_i(x_i)\Lambda(x_i)
\cdot\mathbbm{1}_{V(F)}(\x)
= c(F;\omega,\mathbf{x}_0)P^{n-d}
+O\left(\frac{P^{n-d}}{(\log P)^A}\right),
\]
where $c(F;\omega,\mathbf{x}_0)>0$ is a constant depending only on $F, \omega$ and $\x_0$.
\end{theorem}

\section{Weyl sums over primes}\label{s:terminal}
In this section, we collect several estimates for Weyl sums that follow as consequences of Vinogradov's mean value theorem,
proved by Bourgain, Demeter, and Guth~\cite{bdg} and by Wooley~\cite{wooley-cubic, wooley-vmvt}. Weyl sums over primes
have been extensively studied in the context of the Waring--Goldbach problem, and the relevant estimates
were generalised, from the phase function being of the form $\theta x^d$ to a general polynomial, in~\cite[Section~6]{liuzhao}.

\begin{lemma}\label{lem:Weyl}
Let $d\geq3$, $X\geq1$, and let
\[
 f(x)=\theta x^d+\theta_{d-1}x^{d-1}+\cdots+\theta_1x\in\RR[x].
\]
Let $a \in\ZZ$ and $q \in \NN$ satisfy $\gcd(a,q)=1$.
Then
\begin{equation*}
 \sum_{1 \leq x \leq X} e(f(x))
 \ll X^{1+\ve}
 \bigl(X^{-1}+\lambda^{-1}+\lambda X^{-d}\bigr)^{ \frac{1}{d(d-1)}},
\end{equation*}
where $\lambda= q+ X^d|q\theta-a|$, for any $\ve > 0$. Here the implicit constant is independent of $\theta_{d-1}, \cdots, \theta_{1}$.
\end{lemma}
\begin{proof}
As a consequence of Vinogradov's mean value theorem, it is known (for example, this is recorded in~\cite[Lemma~2.4]{browning-yamagishi}) that
given any $a \in \ZZ$ and $q \in \NN$ such that  $\gcd(a,q) = 1$, $|\theta - a/q| \leq q^{-2}$ and $q \leq X^d$, we have
$$
\sum_{1 \leq x \leq X} e(f(x))
\ll X^{1+\ve}
\bigl(X^{-1}+ q^{-1}+ q X^{-d}\bigr)^{ \frac{1}{d(d-1)} }.
$$
The condition $q \leq X^d$ is superfluous, because the bound holds trivially if $q > X^d$.
Furthermore, it can be seen from the proof, given in~\cite[Theorem~5.2]{vaughan} and~\cite[Theorem~1.5]{wooleyMVT}, that the implicit constant is independent of
$\theta_{d-1}, \cdots, \theta_{1}$. Thus, the result follows on applying~\cite[Lemma~14.1]{wooley}.
\end{proof}

\begin{remark}
By simply shifting the polynomial $f$, we can see that the bound in this lemma holds even if the range of summation $[1, X]$ is replaced with any other interval of length $\ll X$.
\end{remark}

The following Type II estimate can be obtained in a similar manner to~\cite[Lemma~3.1]{KawaWoo} and~\cite[Lemma~6.6]{liuzhao}, whose methods in turn go back to~\cite[\S~3]{vaughan1}, with Lemma~\ref{lem:Weyl} incorporated into the argument.
We include the proof in Appendix~\ref{A} for completeness.
\begin{lemma}
\label{TypeII}
Let $d\geq3$, and let
\[
f(x)=\theta x^d+\theta_{d-1}x^{d-1}+\cdots+\theta_1x\in\RR[x].
\]
Let $\gamma > 0$, $1\leq Q\leq P^{1/4}$, $P \ll X \leq P$, $P^{1/8} \leq M, N \leq P$ and $P \ll M N \ll P$.
Let us define
$$
\mathcal{S} = \sum_{ N \leq  y <  2 N } \sum_{ \substack{ M \leq x <  2 M \\ \gamma P \leq  x y \leq  X } } \xi_1(x) \xi_2(y) e( f(xy) ),
$$
where $|\xi_1(m)|, |\xi_2(m)| \ll m^{\varepsilon}$ for all $m \in \NN$ and any $\ve > 0$. If $\theta\notin\mathfrak M(Q)$,
then
$$
\mathcal{S} \ll  P^{1 + \ve} Q^{- \frac{1}{2 d (d-1)}},
$$
for any $\ve > 0$.
\end{lemma}

Next we prove the following estimate for Weyl sums over primes.
\begin{lemma}\label{lem:prime-Weyl}
Let $d\geq3$, and let
\[
f(x)=\theta x^d+\theta_{d-1}x^{d-1}+\cdots+\theta_1x\in\RR[x].
\]
Let $\gamma > 0$.
If $\theta\notin\mathfrak M(Q)$, $1\leq Q\leq P^{1/4}$ and $P \ll X \leq P$, then
\begin{equation*}
\sum_{\gamma P \leq x\leq X}\Lambda(x)e(f(x)) \ll P^{1+\ve}Q^{- \frac{1}{2d(d-1)} },
\end{equation*}
for any $\ve > 0$.
\end{lemma}

\begin{proof}
By Dirichlet's theorem, there exist $a\in\ZZ$ and $q \in \NN$ such that $\gcd(a,q)=1$,
\[
1\leq q\leq P^{\frac{d}{2}}
\qquad
\textnormal{and}
\qquad
|q\theta-a|\leq P^{- \frac{d}{2}}.
\]
Since $\theta\notin\mathfrak M(Q)$, we have
\begin{equation}\label{eq:mixed-height}
 Q<q+P^d|q\theta-a|\leq2P^{\frac{d}{2}}.
\end{equation}
Let us define
\[
 A(y)=\sum_{\substack{rs=y\\ 1 \leq r,s\leq U}}\mu(r)\Lambda(s)
\qquad
\textnormal{and}
\qquad
 B(y)=\sum_{\substack{r\mid y\\ 1 \leq  r\leq U}}\mu(r),
\]
where $U=P^{1/8}$. We apply Vaughan's identity~\cite[Section~24]{davenport} to decompose the sum into Type I and Type II sums as follows
\begin{equation}
\label{SSS+++}
\sum_{\gamma P \leq x \leq X}\Lambda(x)e(f(x)) = T_{\mathrm{I}}-T_{\mathrm{II}},
\end{equation}
where
\[
 T_{\mathrm{I}}=\sum_{1 \leq y \leq U}\mu(y)
 \sum_{ \gamma P/y \leq x \leq X/y   }(\log x)e(f(xy))- \sum_{1 \leq y \leq U^2} A(y)
 \sum_{ \gamma P/y \leq x \leq X/y  }e(f(xy))
\]
and
\[
T_{\mathrm{II}}=\sum_{ U <  y  \leq X } B(y) \sum_{  \substack{ U <  x \leq X /y   \\  \gamma P \leq xy  } }\Lambda(x)e(f(xy)).
\]
Here $|A(y)|\leq\log y$ and $|B(y)|\leq\tau(y)$, where $\tau(y)$ is the number of divisors of $y$.

For each $1 \leq y \leq U^2$, write $ay^d / q = b_y/r_y$ with $r_y \geq 1$ and $\gcd(b_y,r_y)=1$.
Let
\[
 \lambda_y = r_y +   \frac{ P^d |r_y  y^d \theta - b_y |}{y^d} =\frac{q+P^d|q\theta-a|}{\gcd(q,y^d)}.
\]
Let $0 < \sigma \leq 1 / d$ and $\iota \in \{1, 2 \}$. It follows from \eqref{eq:mixed-height} that
\begin{equation}\label{eq:mixed-height-moments1--}
\lambda_y \leq \frac{q  +   P^{d} |q \theta  - a|  }{\gcd(q, y^d)} \leq 2 P^{\frac{d}{2}}.
\end{equation}
Similarly, since it follows from \eqref{eq:mixed-height} that
$$
\frac{Q}{2} < q \qquad \textnormal{or} \qquad \frac{Q}{2} < P^d |q  \theta  - a|,
$$
we have
\begin{eqnarray}
\label{eq:mixed-height-moments2--}
\sum_{ M \leq y < 2 M }\lambda_y^{-\sigma}
&\ll&
\sum_{ M \leq y < 2 M } \frac{\gcd(q, y^d)^{\sigma}}{ Q^{\sigma} }
\\
\notag
&\ll&
\sum_{\ell |  q}  Q^{- \sigma}  \sum_{g | \ell}  \sum_{ \substack{ M \leq y < 2 M  \\ \gcd(q, y^d) = \ell  \\ \gcd(q, y) = g  } }
\ell^{\sigma}
\\
\notag
&\ll&
\tau(q) \sum_{\ell |  q}  Q^{- \sigma} \ell^{\sigma} \frac{M}{ \ell^{1/d} }
\\
\notag
&\ll&
M Q^{- \sigma} \tau(q)^2.
\end{eqnarray}
Therefore, by Lemma~\ref{lem:Weyl} (with partial summation when $\iota = 1$), \eqref{eq:mixed-height-moments1--} and \eqref{eq:mixed-height-moments2--}, we obtain
\begin{eqnarray}
\label{typeI}
T_{\mathrm{I}} &\ll&  \sum_{\iota \in \{1, 2\} } (\log P) \sum_{1 \leq y \leq U^{\iota}}
\left|   \sum_{ \gamma P /y \leq x \leq X/y} (\log x)^{2 - \iota}  e(f(xy)) \right|
\\
\notag
&\ll&
(\log P)^2 \sum_{1 \leq y \leq U^2 }
\frac{X^{1+ \varepsilon}}{ y^{1 + \ve }}  \left(   \frac{y}{X} + \lambda_y^{-1}+\lambda_y \frac{y^d}{X^d}  \right)^{ \frac{1}{d(d-1)}}
\\
\notag
&\ll&
(\log P)^2
X^{1 - \frac{1}{d(d-1)} +  \ve} U^{ \frac{ 2}{d(d-1)}  }
\\
\notag
&&
+ (\log P)^3  X^{1 + \varepsilon}
\sup_{ \substack{  [M, 2M) \subseteq [1, U^2] \\ M \geq 1 }}
\bigl( Q^{- \frac{1}{d(d-1)}} \tau(q)^2 +
M^{\frac{1}{d-1}} X^{- \frac{1}{d-1}} P^{\frac{1}{2 (d-1)}} \bigr)
\\
\notag
&\ll& P^{1+\ve} \bigl(    Q^{-  \frac{1}{d(d-1)}   } +  P^{-  \frac{1}{4(d-1)} }    \bigr).
\end{eqnarray}

We treat $T_{\mathrm{II}}$ by decomposing it into dyadic sums as follows
\begin{eqnarray*}
T_{\mathrm{II}}
&\ll&
(\log P)^2 \sup \sum_{ \substack{ N \leq y < 2 N  \\  U <  y  \leq X  }} B(y) \sum_{ \substack{  M \leq x < 2 M  \\   U <  x \leq X /y \\ \gamma P \leq x y } }\Lambda(x) e(f(xy)),
\end{eqnarray*}
where the supremum is over $[M,2M)$ and  $[N, 2N)$ with  $U \leq M, N \leq X$ and $P \ll M N \ll P$.
Then, by Lemma~\ref{TypeII}, we obtain
$$
T_{\mathrm{II}} \ll  P^{1 + 2 \ve} Q^{- \frac{1}{2 d (d-1)}}.
$$
Finally, on substituting this estimate and \eqref{typeI} into \eqref{SSS+++}, it follows that
$$
\sum_{ \gamma P  \leq x \leq X}\Lambda(x) e(f(x)) \ll  |T_{\mathrm{I}}| + |T_{\mathrm{II}}|  \ll  P^{1 +  \ve} Q^{- \frac{1}{2 d (d-1)}}.
$$
\end{proof}

\subsection{The pure monomial case}\label{ss:pure}
Let $d\geq3$, let $c\in\ZZ\setminus\{0\}$ be fixed, and let $\Phi \in \RR[x]$ of degree $d$ with the leading coefficient $c \alpha$.
We fix a smooth function $\omega_1: \RR \to \RR$ with $\supp \omega_1 \subseteq (0,1)$. 
\begin{lemma}\label{lem:pure-terminal}
Suppose  $P^{1/100}\leq L\leq P^{d/2}$ and $\alpha\in\mathfrak n(L)$. Then
\begin{equation*}
 \sum_{1 \leq x \leq P} \omega_1 \left( \frac{x}{P} \right) \Lambda(x) \,  e\bigl(\Phi(x)\bigr)
\ll P^{1 + \ve} L^{-  \frac{1}{ 4d^2(d-1)  } },
\end{equation*}
for any $\ve > 0$. Here the implicit constant is independent of the coefficients of $\Phi(x) - c \alpha x^d$.
\end{lemma}

\begin{proof}
Since $\alpha\notin\mathfrak M(L)$, we have $c \alpha \notin \mathfrak M(L/|c|)$.
If $P^{1/4} \leq L/|c|$, then it follows that $c \alpha \notin \mathfrak M(P^{1/4})$.
Let $Q = \min\{L/|c|,P^{1/4}\}$. Let $\gamma > 0$ be sufficiently small.
Therefore, by partial summation and Lemma~\ref{lem:prime-Weyl}, we have
\begin{eqnarray*}
\left| \sum_{1 \leq x \leq P} \omega_1 \left( \frac{x}{P} \right) \Lambda(x) \,  e\bigl(\Phi(x)\bigr) \right|
&=&
P^{-1} \left| \int_{1}^P  \omega_1' \left( \frac{t}{P} \right)   \sum_{\gamma P  \leq x \leq t}  \Lambda(x) \,   e\bigl(\Phi(x)\bigr) \d t \right|
\\
&\ll&
P^{1 + \ve} Q^{- \frac{1}{2d(d-1)}}
\\
&\ll&
\begin{cases}
P^{1 + \ve} L^{- \frac{1}{2d(d-1)}} & \mbox{if } P^{1/100}  \leq L  \leq |c|P^{1/4},  \\
P^{ 1 + \ve  -  \frac{1}{8d(d-1)}} & \mbox{if } |c|P^{1/4}< L \leq P^{d/2}.
\end{cases}
\end{eqnarray*}
In either case, we obtain the bound
$$
\ll P^{1 + \ve} L^{-  \frac{1}{ 4d^2(d-1)  } },
$$
which proves the lemma.
\end{proof}

\subsection{The mixed monomial case}\label{ss:mixed}
Let $d\geq4$, let $c\in\ZZ\setminus\{0\}$ be fixed, and let $\Phi\in\RR[x,y]$ of total degree $d$ with $\deg_x \Phi = d - 1$
and $c\alpha$ the coefficient of $x^{d-1}y$.
We fix smooth functions $\omega_1,\omega_2: \RR \to \RR$ with $\supp \omega_1, \supp \omega_2 \subseteq (0, 1)$.  
\begin{lemma}\label{lem:mixed-terminal}
Suppose $P^{1/100}\leq L\leq P^{d/2}$ and $\alpha\in\mathfrak n(L)$.  Then
\begin{equation*}
 \sum_{ 1 \leq  x \leq P }\sum_{ 1 \leq  y \leq P}
 \omega_1 \left( \frac{x}{P} \right) \omega_2 \left( \frac{y}{P} \right) \Lambda(x) \, \Lambda(y) \, e\bigl(\Phi(x,y)\bigr)
 \ll P^{2 + \ve} L^{- \frac{1}{d(d-1)(d-2)} },
\end{equation*}
for any $\ve > 0$. Here the implicit constant is independent of the coefficients of $\Phi(x, y) - c \alpha x^{d-1}y$.
\end{lemma}

\begin{proof}
Since $\alpha\in\mathfrak M(2L)$, there exist $a_0 \in \ZZ$ and $q_0 \in \NN$ such that $\gcd(a_0,q_0)=1$, $1\leq q_0\leq2L$ and $P^d|q_0\alpha-a_0|\leq2L$.
Write $ca_0/q_0=a/q$ with $q\geq1$ and $\gcd(a,q)=1$. 
Then it follows that
$$
1 \leq q \leq 2L \qquad \textnormal{and} \qquad P^d|q c \alpha - a|\leq 2 |c| L.
$$
Since $c \alpha \notin \mathfrak M(L/|c|)$, we have
\begin{equation}\label{eq:mixed-height+}
\frac{L}{|c|} < q +  P^d |q c \alpha  - a| \ll L.
\end{equation}
For each $1\leq h\leq P$, write $ah/q = b_h/r_h$ with $r_h\geq1$ and $\gcd(b_h,r_h)=1$, and put
\[
\lambda_h=r_h+P^{d-1}|r_h h c \alpha - b_h|
 =\frac{q  +   hP^{d-1} |q c \alpha   - a|  }{\gcd(q,h)}.
\]

Let $0 < \sigma < 1$. It follows from \eqref{eq:mixed-height+} that
$$
\lambda_h \leq \frac{q  +   P^{d} |q c \alpha   - a|  }{\gcd(q,h)} \ll L.
$$
Therefore, we have
\begin{equation}\label{eq:mixed-height-moments1}
\sum_{1 \leq h\leq P} \lambda_h^\sigma \ll P L^\sigma.
\end{equation}
Similarly, since it follows from \eqref{eq:mixed-height+} that
$$
\frac{L}{2|c|} < q \qquad \textnormal{or} \qquad \frac{L}{2|c|} < P^d |q c \alpha  - a|,
$$
we obtain
\begin{eqnarray}
\label{eq:mixed-height-moments2}
\sum_{1 \leq h\leq P}\lambda_h^{-\sigma}
&\ll&
\sum_{1 \leq h\leq P}
\left( \frac{\gcd(q,h)^{\sigma}}{L^{\sigma}} + \frac{\gcd(q,h)^{\sigma} P^{\sigma}}{ h^{\sigma} L^{\sigma}  } \right)
\\
\notag
&\ll&
\sum_{\ell |  q} \ell^{\sigma} L^{- \sigma} \sum_{ \substack{ 1 \leq h\leq P \\ \gcd(h, q) = \ell  } }
\left( 1 + \frac{ P^{\sigma}}{ h^{\sigma} } \right)
\\
\notag
&\ll&
\sum_{\ell | q } \ell^{\sigma} L^{- \sigma } \frac{P}{\ell }
\\
\notag
&\ll& P L^{-\sigma} \tau(q)
\\
\notag
&\ll& P^{1 + \varepsilon} L^{-\sigma}.
\end{eqnarray}

Let
\[
S = \sum_{ 1 \leq  x \leq P }\sum_{ 1 \leq  y \leq P} \omega_1 \left( \frac{x}{P} \right) \omega_2 \left( \frac{y}{P} \right) \Lambda(x) \, \Lambda(y) \, e\bigl(\Phi(x,y)\bigr).
\]
Let $\gamma > 0$ be sufficiently small.
By the Cauchy--Schwarz inequality, we obtain
\begin{eqnarray*}
|S|^2 &\ll&  P (\log P)^2  \sum_{ \gamma P \leq  x \leq P } \sum_{ 1 \leq  y_1, y_2 \leq P}
\\
&&  \omega_2 \left( \frac{y_1}{P} \right) \omega_2 \left( \frac{y_2}{P} \right) \Lambda(y_1) \, \Lambda(y_2) \, e\bigl(\Phi(x,y_1) - \Phi(x,y_2)\bigr)
\\
&\ll& P^3 (\log P)^4 +  P (\log P)^4 \sum_{ \substack{ 1 \leq y_1,  y_2 \leq P \\ y_1 \neq y_2  } }
\\
&&
\left| \sum_{\gamma P  \leq x \leq P}  e\bigl(  c \alpha (y_1 - y_2) x^{d-1} + \nu(x, y_1, y_2)  \bigr) \right|,
\end{eqnarray*}
where
$$
\nu(x, y_1, y_2) = \Phi(x,y_1) -  c \alpha y_1  x^{d-1}  - \Phi(x,y_2) + c \alpha y_2  x^{d-1}.
$$
In particular, $\deg_x \nu \leq d - 2$. Therefore, it follows from Lemma~\ref{lem:Weyl}, \eqref{eq:mixed-height-moments1} and \eqref{eq:mixed-height-moments2} that
\begin{eqnarray*}
|S|^2 &\ll&  P^3 (\log P)^4 +  P^2 (\log P)^4 \sum_{1 \leq h \leq P} P^{1+ \varepsilon}  \bigl(P^{-1}+\lambda_h^{-1}+\lambda_h P^{-d + 1}\bigr)^{ \frac{1}{(d-1)(d-2)}}
\\
&\ll&
P^{3 + 2 \varepsilon}
\left( P^{1 - \frac{1}{(d-1)(d-2)} } + P^{1 + \varepsilon}  L^{- \frac{1}{(d-1)(d-2)}} + L^{ \frac{1}{(d-1)(d-2)}} P^{ 1 -  \frac{1}{ d-2 }} \right)
\\
&\ll&
P^{4 + 2 \varepsilon} L^{- \frac{2}{d (d-1)(d-2) }},
\end{eqnarray*}
where the final inequality follows on recalling that $P^{1/100}\leq L\leq P^{d/2}$ and $d\geq4$.
The lemma then follows by taking the square roots.
\end{proof}

\section{The minor arcs}
In this section, we prove the following estimate over the minor arcs.
\begin{theorem}
\label{t:minor}
Let $F \in \ZZ[x_1, \ldots, x_n]$ be a homogeneous form of degree $d \geq 4$ with $V_F^{*} = \{ \mathbf{0} \}$.
Suppose
$$
n \geq (24 (d-1)^2 d^2 + 7 (d-1) ) 2^{d} + 11.
$$
Let $S(\alpha)$ be as in \eqref{defS}, and let $Q_0 = P^{1/100}$. Let $\mathfrak{a} \subseteq \RR / \ZZ$ be a measurable set.
Then there exists $\delta > 0$ such that
$$
\int_{ \mathfrak{a} \cap  ((\RR / \ZZ) \setminus \mathfrak{M}(Q_0) ) } S(\alpha) \d \alpha \ll P^{n - d- \delta}.
$$
\end{theorem}

Let us partition the variables $\x$ as
$$
\x = (\y, \z, \w),
$$
where $\y \in \ZZ^{m}$, $\z \in \ZZ^{s}$, $\w \in \ZZ^{t}$ with $m + s + t = n$.
Then $F$ can be uniquely represented as
\begin{eqnarray}
\label{deffg}
F(\y, \z , \w) = f(\y) + g(\y, \z) + h(\y, \z, \w),
\end{eqnarray}
where
$$
\deg_\y g < d \qquad \textnormal{and} \qquad  \deg_{\y, \z} h < d.
$$
Without loss of generality we let $\y = (x_1, \cdots, x_m)$ and $\w = (x_{m+s+1}, \dots, x_n)$.
We define
\begin{eqnarray*}
\mathcal{E}_{\y, \z} (\alpha) = \sum_{ \w \in \ZZ^t } \psi ( \w )   \Lambda(\w)   e (\alpha h(\y, \z, \w)),
\end{eqnarray*}
where
$$
\psi ( \w ) \Lambda(\w) = \prod_{i = 1}^t \psi_{m + s + i} (x_{m+s+i})   \Lambda(x_{m+s+i}).
$$

The following mean value estimate is the key ingredient established in~\cite{liuzhao}.
\begin{proposition}\cite[Proposition~5.1]{liuzhao}
\label{LZ}
Let $S(\alpha)$ be as in \eqref{defS} and $f, g$ as in \eqref{deffg}.
Let $X \leq \frac14 P^{ \frac{d}{2} }$. Let $\kappa_1$ and $\kappa_2$ be real numbers such that
$$
\min (\kappa_1, \kappa_2)  > (d-1) 2^d.
$$
Suppose
$$
m - \dim V_f^*  \geq \kappa_1 \qquad
\textnormal{and}
\qquad
m + s - \dim V_g^*  \geq \kappa_2.
$$
Let $\mathfrak{a} \subseteq \RR / \ZZ$ be a measurable set.
Then we have
\begin{eqnarray*}
\int_{\mathfrak{a}} S(\alpha) \d \alpha  &\ll&  P^{n - \frac{d}{2}}   (\log P)^{2n}  X^{1 - \frac{   \kappa_1 }{ (d-1) 2^d  } }  |\mathfrak{a}|^{\frac12}
\\
&&+ P^{ n - \frac{3d}{4} }   (\log P)^{2n}  X^{\frac{3}{2} - \frac{   \kappa_2 }{ (d-1) 2^{d+1}  } }  |\mathfrak{a}|^{\frac14}
\\
&&+ P^{ m + s - d }   (\log P)^{m + s}  X^{2} \sup_{\alpha \in \mathfrak{a}} \sup_{ (\y, \z) \in  [0, P]^{m + s} \cap \ZZ^{m + s}  } |  \mathcal{E}_{\y, \z} (\alpha) |.
\end{eqnarray*}
\end{proposition}

\begin{remark}
The assumption that $F$ defines a smooth hypersurface is not necessary for this proposition.
Also, the weights of the exponential sum $S(\alpha)$ satisfy
$$
| \psi_i(x) \Lambda (x) | \leq \log x \qquad (x \in [1,P] \cap \ZZ)
$$
for all $1 \leq i \leq n$, which is sufficient to apply~\cite[Proposition~5.1]{liuzhao} $($see~\cite[pp.8625]{liuzhao}$)$.
\end{remark}

Let us define
$$
\varrho = \frac{1}{4 d^2 (d-1)}.
$$
Given $L \geq 1$, let
$$
W(L) = \sup_{\alpha \in \mathfrak{n}(L) } \sup_{ (\y, \z) \in  [0, P]^{m + s} \cap \ZZ^{m + s}  } |  \mathcal{E}_{\y, \z} (\alpha) |,
$$
where $\mathfrak{n}(L)$ is defined in \eqref{defnL}.
\begin{lemma}
\label{LEMMA2}
Suppose
$$
m - \dim V_f^*  > \left(  \frac{1}{\varrho} + \frac{1}{2}  \right)   (d-1) 2^{d + 1}
$$
and
$$
m + s - \dim V_g^*  > \left(  \frac{1}{2 \varrho} +  \frac{3}{4}  \right)  (d-1) 2^{d + 2}.
$$
Suppose further that, for any $\varepsilon_0 > 0$ sufficiently small, we have
$$
W(L) \ll P^{t + \varepsilon_0} L^{- \varrho}
$$
for any $P^{1/100} \leq L \leq P^{d/2}$. Then there exists $\delta > 0$ such that
$$
\int_{ \mathfrak{a} \cap \mathfrak{n}(L) } S(\alpha) \d \alpha \ll P^{n - d }   (\log P)^{2n}  L^{-\delta}
$$
holds for any measurable $\mathfrak{a} \subseteq \RR/ \ZZ$.
\end{lemma}

\begin{proof}
Let $\delta_0 > 0$ be sufficiently small, and let $0 <  \varepsilon_0 <  \delta_0 / 100$.
We set
$$
X =  L^{\frac{\varrho}{2}  - \delta_0 } \leq \frac{1}{4}P^{\frac{d}{2}}.
$$
Let
$$
m - \dim V_f^*  \geq \kappa_1 \qquad
\textnormal{and}
\qquad
m + s - \dim V_g^*  \geq \kappa_2
$$
for some $\kappa_1, \kappa_2 > (d-1) 2^d$.
Let us recall from \eqref{measnL} that
$$
|\mathfrak{n}(L)| \ll L^{2} P^{ -d}.
$$
Then it follows from Proposition~\ref{LZ} that
\begin{eqnarray}
\int_{ \mathfrak{a} \cap \mathfrak{n}(L) } S(\alpha) \d \alpha  &\ll&  P^{n - \frac{d}{2}}   (\log P)^{2n}  X^{1 - \frac{   \kappa_1 }{ (d-1) 2^d  } } L P^{- \frac{d}{2}}
\label{mean}
\\
\notag
&&+ P^{ n - \frac{3d}{4} }   (\log P)^{2n}  X^{\frac{3}{2} - \frac{   \kappa_2 }{ (d-1) 2^{d+1}  } } L^{\frac{1}{2}} P^{- \frac{d}{4}}
\\
\notag
&&+ P^{ m + s - d }   (\log P)^{m + s}  X^{2}  W(L).
\end{eqnarray}
Therefore, if
$$
\frac{ \varrho  \kappa_1 }{ (d-1) 2^{d+1}  }  >  1 +  \frac{\varrho}{2}
\qquad
\textnormal{and}
\qquad
\frac{ \varrho  \kappa_2 }{ (d-1) 2^{d+2} }  >  \frac{1}{2} + \frac{3 \varrho}{4},
$$
or equivalently,
$$
\kappa_1 >   \left(  \frac{1}{\varrho} + \frac{1}{2}  \right)   (d-1) 2^{d + 1}
\qquad
\textnormal{and}
\qquad
\kappa_2  >   \left(  \frac{1}{2 \varrho} +  \frac{3}{4}  \right)  (d-1) 2^{d + 2},
$$
then \eqref{mean} becomes
\begin{eqnarray*}
\int_{ \mathfrak{a} \cap   \mathfrak{n}(L) } S(\alpha) \d \alpha
&\ll&
P^{n - d }    (\log P)^{2n}  L^{-\varepsilon} + P^{ n - t - d }   (\log P)^{n - t}  X^{2}  W(L)
\\
&\ll&
P^{n - d }    (\log P)^{2n}  L^{-\varepsilon} + P^{ n - t - d }   (\log P)^{n - t}  L^{\varrho - 2 \delta_0}  W(L)
\\
&\ll&
P^{n - d }    (\log P)^{2n}  L^{-\varepsilon} + P^{ n - d }   (\log P)^{n - t}  L^{- 2 \delta_0}  P^{\varepsilon_0}
\\
&\ll&
P^{n - d }    (\log P)^{2n}  L^{-\varepsilon},
\end{eqnarray*}
for some $\varepsilon > 0$.
\end{proof}

We make use of the following simple structure regarding $F$.
\begin{lemma}
\label{smooth}
Let $F \in \ZZ[x_1, \dots, x_n]$ be a homogeneous form of degree $d \geq 2$ such that $V_F^* = \{ \mathbf{0} \}$.
Then either there exists $1 \leq i \leq n$ such that $x_i^d$ has a nonzero coefficient in $F$,
or there exist $1 \leq j, k \leq n$ such that $x_j^{d-1} x_k$ has a nonzero coefficient.
\end{lemma}
\begin{proof}
For each $2 \leq k \leq n$, let $C_k$ denote the coefficient of $x_{1}^{d-1} x_k$ in $F$.
Then
$$
\frac{ \partial F}{ \partial x_k }(1, 0, \ldots, 0) = \begin{cases}
                                                       d F(1, 0, \ldots, 0)  & \mbox{if } k = 1,  \\
                                                       C_k  & \mbox{if } 2 \leq k \leq n.
                                                      \end{cases}
$$
Therefore, if the coefficients of $x_i^d$  and $x_j^{d-1} x_k$  in $F$ are zero for all $1 \leq i, j, k \leq n$,
then this leads to
$$
(1, 0, \ldots, 0) \in V_F^*,
$$
which is a contradiction.
\end{proof}

\begin{proof}[Proof of Theorem~\ref{t:minor}]
We begin by partitioning the variables in a way that Lemma~\ref{LEMMA2} is applicable.
If we are in the first case of Lemma~\ref{smooth}, i.e.  there exists $1 \leq i \leq n$ such that $x_i^d$ has a nonzero coefficient in $F$,
then we assume without loss of generality that $i = n$, and set $\w = x_n$. In this case, it follows from
Lemma~\ref{lem:pure-terminal} that
$$
W(L) \ll P^{t + \varepsilon_0} L^{- \varrho},
$$
for any $P^{1/100} \leq L \leq P^{d/2}$. On the other hand, if we are not in the first case of Lemma~\ref{smooth},
i.e. there exist $1 \leq j, k \leq n$ such that $x_j^{d-1} x_k$ has a nonzero coefficient in $F$
while the coefficient of $x_i^d$ is zero for all $1 \leq i \leq n$, then we assume without loss of generality
that $(j, k) = (n-1, n)$, and set $\w = (x_{n-1}, x_n)$. In this case, it follows from
Lemma~\ref{lem:mixed-terminal} that
$$
W(L) \ll P^{t + \varepsilon_0} L^{- \varrho},
$$
for any $P^{1/100} \leq L \leq P^{d/2}$. Thus, we see that in either case the upper
bound on $W(L)$ required in Lemma~\ref{LEMMA2} is satisfied.

Since $V_F^*  = \{ \mathbf{0} \}$, it follows from~\cite[Lemma~3.1]{yamagishi} that
$$
\textnormal{codim }V_f^* \geq \textnormal{codim }V_F^* - 2(s + t) = n - 2 (s + t).
$$
Let
$$
V_1 = \left\{ \x \in \AA^n_\CC:  \frac{\partial F }{\partial x_{1}} (\x) = \cdots = \frac{\partial F }{\partial x_{m}} (\x)  = 0  \right\}
$$
and
$$
V_2 = \left\{ \x \in \AA^n_\CC:  \frac{\partial F }{\partial x_{m+1}} (\x)  = \cdots = \frac{\partial F }{\partial x_{n}} (\x) = 0  \right\}.
$$
Since these are affine varieties defined by homogeneous forms, it follows that
$$
0 = \dim V_F^* = \dim (V_1 \cap V_2) \geq \dim V_2 - m,
$$
i.e.
$$
\dim V_2 \leq m.
$$
On noting that
$$
\frac{\partial (F - f) }{ \partial x_i } (\x) = \frac{\partial F  }{ \partial x_i } (\x) \qquad (m + 1 \leq i \leq n),
$$
we have
$$
V_{F - f}^* \subseteq V_2.
$$
Hence,
$$
\textnormal{codim } V_{F - f}^* = n - \dim V_{F - f}^* \geq n - m.
$$
Therefore, it again follows from~\cite[Lemma~3.1]{yamagishi}  that
$$
\textnormal{codim } V_{g}^*  \geq \textnormal{codim } V_{F - f}^* - 2 t \geq n - m - 2t.
$$
Thus, the dimension requirements of Lemma~\ref{LEMMA2} are satisfied if
$$
n - 2(s + t)  \geq \left(  \frac{1}{\varrho} + \frac{1}{2}  \right)   (d-1) 2^{d + 1} + 1
$$
and
$$
n - m - 2t  =   (s + t) - 2t  \geq \left(  \frac{1}{2 \varrho} +  \frac{3}{4}  \right)  (d-1) 2^{d + 2} + 1.
$$
This means that provided
\begin{eqnarray}
\label{eq:nnn}
n - 4 t  &=& n - 2(s + t) + 2 (s + t) - 4t
\\
\notag
&\geq&
\left(  \frac{1}{\varrho} + \frac{1}{2}  \right)   (d-1) 2^{d + 1}
+ \left(  \frac{1}{\varrho} +  \frac{3}{2}  \right)  (d-1) 2^{d + 2} + 3
\\
\notag
&=&   \left(  \frac{6}{\varrho} +  7\right)     (d-1) 2^{d} + 3,
\end{eqnarray}
we may achieve\footnote{We can choose
$$
s  =    \left(  \frac{1}{2 \varrho} +  \frac{3}{4}  \right)  (d-1) 2^{d + 2} + t + 1 \qquad \textnormal{and} \qquad m = n - s - t.
$$}
the desired partition of the variables. Finally, since $t \in \{1, 2\}$, it is easy to see that \eqref{eq:nnn} holds given
$$
n \geq ( 24 (d-1) d^2 + 7) (d-1) 2^{d} + 11 = (24 (d-1)^2 d^2 + 7 (d-1) ) 2^{d} + 11.
$$

Let us recall from \eqref{dydshell} that
$$
(\RR / \ZZ) \setminus \mathfrak{M} (Q_0)
=  \bigcup_{1 \leq j \leq J} \mathfrak{n}(2^{j-1} Q_0).
$$
Therefore, given any measurable $\mathfrak{a} \subseteq \RR/\ZZ$, by applying Lemma~\ref{LEMMA2}, we obtain
\begin{eqnarray*}
\int_{ \mathfrak{a} \cap ((\RR / \ZZ) \setminus \mathfrak{M} (Q_0) ) } S(\alpha) \d \alpha &=&  \sum_{j = 1}^{J} \int_{ \mathfrak{a} \cap  \mathfrak{n} ( 2^{j-1} Q_0)}   S(\alpha) \d \alpha
\\
&\ll&  P^{n - d }    (\log P)^{2n}  \sum_{j = 1}^J   (2^{j-1} Q_0)^{-\varepsilon}
\\
&\ll&  P^{n - d - \delta},
\end{eqnarray*}
where $\delta > 0$.
\end{proof}

\section{The major arcs}
Let $F \in \mathbb{Z}[x_1, \ldots, x_n]$ be a homogeneous form of degree $d \geq 2$. 
Let $\mathbf{x}_0$, $\delta$, $\omega$ and $\psi_1, \ldots, \psi_n$ be as in Section~\ref{s:prelim}.
Suppose further that the equation $F(\x) = 0$ has a nonsingular solution in $( \mathbb{Z}_p^{\times})^n$ for every prime $p$.

We define the following extended major arcs
\begin{equation}
\label{defn major}
\mathfrak{M}^+(\vartheta_0; \lambda)  = \bigcup_{1 \leq q \leq P^{\vartheta_0}} \bigcup_{\substack{ 0 \leq a <  q  \\ \gcd(a, q) = 1}} \mathfrak{M}^+_{q,a}(\vartheta_0),
\end{equation}
where
$$
\mathfrak{M}^+_{q, a} (\vartheta_0) = \left\{ {\alpha} \in \RR/\ZZ :  \left\| \alpha - \frac{a}{q}  \right\| \leq  P^{\vartheta_0 + \lambda  - d} \right\}.
$$
It can be verified that the arcs $\mathfrak{M}^+_{q, a}(\vartheta_0)$ are disjoint for $P$ sufficiently large (see for example~\cite[Lemma~4.1]{birch} for the argument).

We now improve the complete character sum estimate in~\cite[Lemma~2.4]{yamagishi2} using the work of Nguyen~\cite{nguyen1, nguyen}.
\begin{lemma}\label{lem:major-character}
Let $F \in \mathbb{Z}[x_1, \ldots, x_n]$ be a homogeneous form of degree $d \geq 2$.
Let $a \in \mathbb{Z}$ and $q \in \mathbb{N}$  be such that $\gcd(a,q)=1$.
Let $\chi_1, \ldots, \chi_n$ be any Dirichlet characters modulo $q$, and let
\begin{eqnarray}
\label{calS}
\mathcal{S} (\chi_1, \ldots, \chi_n) = \sum_{\mathbf{h} \in (\mathbb{Z}/q \mathbb{Z})^n  } \chi_1(h_1) \cdots \chi_n(h_n)
\ e \left( \frac{a}{q} F(h_1, \ldots, h_n)  \right).
\end{eqnarray}
Then, for any $\varepsilon > 0$, we have
$$
|\mathcal{S} (\chi_1, \ldots, \chi_n) | \ll q^{n - \frac{ \textnormal{codim } V_{F}^* }{16(2d-1)}  + \varepsilon},
$$
where the implicit constant is independent of the choice of $\chi_1, \ldots, \chi_n$.
\end{lemma}

\begin{proof}
We have the following bound recorded in the proof of~\cite[Theorem~D]{nguyen1},
which follows from~\cite[Theorem~C]{nguyen1} (or~\cite[Theorem~C]{nguyen} with the trivial weight $w = (1, \ldots, 1)$ and $r=1$) and the Chinese remainder theorem.
Let $H \in \ZZ[x_1, \ldots, x_N]$ be a homogeneous form of degree $d \geq 2$.
Let $a \in \ZZ$ and  $q \in \NN$ such that $\gcd(a,q) = 1$.
Then, for any $\ve > 0$ sufficiently small, we have
\begin{equation}
\label{NGY}
\sum_{\x \in (\ZZ/ q\ZZ)^N} e \left(  \frac{a H(\x)}{q} \right) \ll q^{N - \frac{\textnormal{codim } V_H^*}{2(d-1)} + \ve }.
\end{equation}

It is shown in the proof of~\cite[Lemma~7.4]{yamagishi}, in~\cite[(7.10)]{yamagishi} to be precise, that
\begin{equation}
\label{CCSS}
|\mathcal{S} (\chi_1, \ldots, \chi_n)|^4 \leq \sum_{\h, \h' \in (\ZZ/ q \ZZ)^n}
\sum_{\kk, \kk' \in (\ZZ/ q \ZZ)^n} e \left( \frac{a}{q} \mathfrak{D}( \h, \h'; \kk, \kk'  )  \right),
\end{equation}
where
$$
\mathfrak{D}( \h, \h'; \kk, \kk'  ) = G(\h; \kk) - G(\h'; \kk) - G(\h; \kk') + G(\h'; \kk')
$$
and
$$
G(\h; \kk) = F(h_1k_1, \ldots h_n k_n).
$$
It is also shown in the same proof that
$$
\textnormal{codim } V_{\mathfrak{D}, 1}^* \geq \frac{\textnormal{codim } V_F^*}{2},
$$
where
\begin{eqnarray*}
&& V_{\mathfrak{D}, 1}^*
\\
&=& \left\{  (\h, \h', \kk, \kk') \in \AA^{4n}_{\CC}: \frac{\partial \mathfrak{D}}{\partial h_i}  ( \h, \h'; \kk, \kk'  )  = \frac{\partial \mathfrak{D}}{\partial h'_i}  ( \h, \h'; \kk, \kk'  )   =   0  \, \,  (1 \leq i \leq n) \right\}.
\end{eqnarray*}
Since
$$
V_{\mathfrak{D}}^* \subseteq V_{\mathfrak{D}, 1}^*,
$$
it follows that
$$
\textnormal{codim } V_{\mathfrak{D}}^* \geq \frac{\textnormal{codim } V_F^*}{2}.
$$
Therefore, by applying \eqref{NGY} to \eqref{CCSS}, we obtain
$$
|\mathcal{S} (\chi_1, \ldots, \chi_n)|^4 \ll q^{4n - \frac{\textnormal{codim } V_F^*}{4(2d-1)} + \ve }.
$$
The lemma then follows by taking the fourth root.
\end{proof}

Let $\vartheta_0, \gamma, \lambda > 0$ be such that
\begin{eqnarray}
\label{theta0 major}
2 \vartheta_0 + \gamma < \frac{5}{12}
\qquad
\textnormal{ and }
\qquad
\gamma > 2 \vartheta_0 + 2 \lambda.
\end{eqnarray}
By replacing~\cite[Lemma~2.4]{yamagishi2} with Lemma~\ref{lem:major-character}, we obtain
the following estimate over the major arcs from~\cite[Proposition~3.1]{yamagishi2}; this proposition refines the argument in~\cite[Section~8]{yamagishi}.
\begin{proposition}\cite[Proposition~3.1]{yamagishi2}
\label{prop major}
Let $F$, $\mathbf{x}_0$, $\delta$, $\omega$ and $\psi_1, \ldots, \psi_n$ be as in this section.
Suppose
$$
\textnormal{codim } V_F^* >  \max \left( 48(2d-1), (d-1) 2^{d}  + 4  (2 d - 1),   \frac{ (d-1) d 2^{d+2} }{  2d-1   }  \right).
$$
Let $\vartheta_0, \gamma, \lambda > 0$ be such that \eqref{theta0 major} is satisfied.
Then, for any $A > 0$, we have
$$
\int_{\mathfrak{M}^+(\vartheta_0; \lambda )} S(\alpha)   \d \alpha = c(F; \omega, \mathbf{x}_0)  P^{n-d} + O \left(  \frac{P^{n-d}}{(\log P)^A} \right),
$$
where $c(F; \omega, \mathbf{x}_0) > 0$ is a constant depending only on $F$, $\omega$ and $\x_0$.
\end{proposition}
The constant $c(F; \omega, \mathbf{x}_0)$ is the usual product of local densities (see~\cite{yamagishi, yamagishi2} for more details).
We are now in position to prove our main theorem.
\begin{proof}[Proof of Theorem~\ref{t:asymptotic}]
Let us set
$$
\vartheta_0 = \frac{1}{16} - \varepsilon,
\qquad
\lambda = \frac{1}{12}
\qquad
\textnormal{and}
\qquad
\gamma = \frac{7}{24}
$$
so that \eqref{theta0 major} is satisfied.
Since
$$
\frac{1}{100} \leq \vartheta_0,
$$
it is clear that
$$
\mathfrak{M} (Q_0) \subseteq  \mathfrak{M}^+ (\vartheta_0; \lambda),
$$
which implies
$$
\mathfrak{a} =  (\RR / \ZZ) \setminus  \mathfrak{M}^+ (\vartheta_0; \lambda) \subseteq   (\RR / \ZZ) \setminus \mathfrak{M} (Q_0).
$$
Therefore, by applying Theorem~\ref{t:minor} and Proposition~\ref{prop major}, we obtain
\begin{eqnarray*}
\sum_{\x \in \ZZ^n}  \prod_{i = 1}^n  \psi_i (x_i)  \Lambda(x_i)  \cdot \mathbbm{1}_{V(F)}(\x)
&=& \int_{\RR/\ZZ}  S(\alpha)   \d \alpha
\\
&=&
\int_{\mathfrak{M}^+(\vartheta_0; \lambda )} S(\alpha)   \d \alpha
+   \int_{ \mathfrak{a} } S(\alpha)   \d \alpha
\\
&=& c(F; \omega, \mathbf{x}_0) P^{n-d} + O \left(  \frac{P^{n-d}}{(\log P)^A} \right).
\end{eqnarray*}
\end{proof}

\appendix
\section{Proof of Lemma~\ref{TypeII}}
\label{A}

Without loss of generality let us suppose $N \leq M$. We begin by applying the Cauchy--Schwarz inequality to obtain the following
\begin{eqnarray}
\label{S^2}
&&|\mathcal{S}|^2
\\
\notag
&\ll& M^{1 + 2 \varepsilon}
\sum_{ M \leq  x < 2 M }  \sum_{ \substack{ N \leq   y_1, y_2  < 2 N \\  \gamma P/x \leq   y_1, y_2  \leq  X/x } } \xi_2(y_1) \overline{\xi_2(y_2)} e(f(x y_1) - f(x y_2))
\\
\notag
&\ll&
M^{2 + 4\ve} N
+
M^{1 + 4 \ve} \sum_{ \substack{ N \leq y_1, y_2 <  2 N \\  y_1, y_2 \leq P/M \\ y_1 - y_2 > 0}}  \left| \sum_{ x \in I(y_1, y_2)  }  e(f(x y_1) - f(x y_2)) \right|,
\end{eqnarray}
where $I(y_1, y_2) = [M, 2M) \cap [\gamma P/ y_1, X/y_1] \cap [\gamma P/ y_2, X / y_2]$.

Let $y_1 > y_2 \geq N$. By Dirichlet's theorem, there exist $a\in\ZZ$ and $q \in \NN$ such that  $\gcd(a,q)=1$,
\[
1\leq q\leq M^{d - 1}
\qquad
\textnormal{and}
\qquad
|q (y_1^d - y_2^d ) \theta-a|\leq M^{- d + 1}.
\]
Let
\[
\lambda_{y_1, y_2} = q + M^{d}|q (y_1^d - y_2^d) \theta - a|.
\]
Since $\lambda_{y_1, y_2}  \leq 2 M^{d-1}$, it follows that
\begin{eqnarray}
\label{bounddd}
\sum_{ \substack{  N \leq  y_1, y_2 <  2 N\\  y_1, y_2 \leq P/M \\ y_1 - y_2 > 0}} \lambda_{y_1, y_2}^{ \frac{1}{d(d-1)}}
\ll P^2  M^{- 2 +  \frac{1}{d}}.
\end{eqnarray}
Therefore, by Lemma~\ref{lem:Weyl} and \eqref{bounddd}, we obtain
\begin{eqnarray*}
&&
\sum_{ \substack{  N \leq  y_1, y_2 <  2 N\\  y_1, y_2 \leq P/M \\ y_1 - y_2 > 0}}  \left| \sum_{x \in I(y_1, y_2) }  e(f(x y_1) - f(x y_2)) \right|
\\
&\ll&
M^{1+ \varepsilon}  \sum_{ \substack{  N \leq  y_1, y_2 <  2 N\\  y_1, y_2 \leq P/M \\ y_1 - y_2 > 0}}  \bigl(   M^{-1} + \lambda_{y_1, y_2}^{-1}+\lambda_{y_1, y_2} M^{-d}  \bigr)^{ \frac{1}{d(d-1)}}
\\
%
&\ll&
P^{2} M^{- 1 -  \frac{1}{d(d-1)} + \varepsilon} +  M^{1 + \ve} E,
\end{eqnarray*}
where
$$
E = \sum_{ \substack{  N \leq  y_1, y_2 <  2 N\\  y_1, y_2 \leq P/M \\ y_1 - y_2 > 0}}
\lambda_{y_1, y_2}^{-  \frac{1}{d(d-1)} } = \sum_{ (y_1, y_2) \in \mathcal{Y} }
\lambda_{y_1, y_2}^{-  \frac{1}{d(d-1)} }.
$$
By substituting this estimate into \eqref{S^2}, we see that it suffices to prove
\begin{equation}
\label{Ebdd}
E \ll  P^{2 +  \ve} M^{ - 2 }  Q^{ - \frac{1}{d(d-1)} }.
\end{equation}

Let $\eta_1 > 0$ be sufficiently small. We have
$$
\sum_{ \substack{ (y_1, y_2) \in \mathcal{Y} \\ q > \eta_1 Q } }  \lambda_{y_1, y_2}^{-  \frac{1}{d(d-1)} }  \ll
\sum_{ \substack{ (y_1, y_2) \in \mathcal{Y} \\ q > \eta_1 Q }} Q^{-  \frac{1}{d(d-1)} } \ll P^2 M^{- 2} Q^{-  \frac{1}{d(d-1)} },
$$
which satisfies \eqref{Ebdd}. Similarly, we also have
$$
\sum_{ \substack{ (y_1, y_2) \in \mathcal{Y} \\ M^{d}|q (y_1^d - y_2^d) \theta - a| > \eta_1 Q}} \lambda_{y_1, y_2}^{-  \frac{1}{d(d-1)} }   \ll P^2 M^{- 2} Q^{-  \frac{1}{d(d-1)} }.
$$
Therefore, from this point on, we may suppose without loss of generality that
$$
q \leq \eta_1 Q \qquad \textnormal{and} \qquad M^{d}|q (y_1^d - y_2^d) \theta - a| \leq  \eta_1 Q  \qquad ( (y_1, y_2) \in \mathcal{Y}).
$$

Let $1 \leq g < 2 N$ and $1 \leq h < N/g$. By Dirichlet's theorem, there exist $b\in\ZZ$ and $r \in \NN$ such that  $\gcd(b,r)=1$,
\[
1 \leq r \leq M^{\frac{d}{2}} N^{\frac{d-1}{2}} g^{- \frac{d-1}{2}}
\qquad
\textnormal{and}
\qquad
| r    g^d  h    \theta - b |\leq  M^{ - \frac{d}{2}} N^{- \frac{d-1}{2}} g^{\frac{d-1}{2}}.
\]
Given ($y_1, y_2) \in \mathcal{Y}$, let us denote
$$
g = \gcd(y_1, y_2),
\qquad
h = \frac{y_1 - y_2}{g}
$$
and
$$
\frac{d N^{d-1}}{g^{d-1}} \leq k = \frac{y_1^d - y_2^d}{ g^{d-1}( y_1 - y_2)} \leq  \frac{d 2^{d-1} N^{d-1}}{g^{d-1}}.
$$
Then we have
\begin{eqnarray*}
\left| \frac{a}{q} - \frac{ k b }{r}  \right|
&\leq& \frac{ | q (y_1^d - y_2^d) \theta - a |  }{q} + \frac{ | r g^d h k  \theta -  k  b |  }{r}
\\
&\leq& \frac{ \eta_1 Q }{ q M^d} +  \frac{ k  M^{ - \frac{d}{2}} N^{- \frac{d - 1}{2}} g^{ \frac{d-1}{2} } }{r}
\\
&\leq& \frac{\eta_1 Q }{ q M^d} +  \frac{ \eta_1  Q  d 2 ^{d-1}  M^{ - \frac{d}{2}} N^{ \frac{d-1}{2} } }{q r}
\\
&<& \frac{ 1 }{qr},
\end{eqnarray*}
where the final inequality follows since $d \geq 3$, $N \leq M$, $Q^4 \leq P \ll M N$ and
$$
\frac{ \eta_1 r  Q }{  M^d} +  \eta_1  Q  d 2 ^{d-1}  M^{ - \frac{d}{2}} N^{ \frac{d-1}{2} }  \leq  \eta_1 (1 + d 2^{d-1}) Q  M^{ - \frac{d}{2} } N^{ \frac{d-1}{2} } <  1.
$$
Therefore, it follows that
$$
\frac{a}{q}  =  \frac{ k b }{r},
$$
which implies
\begin{eqnarray*}
\lambda_{y_1, y_2} &=&   q + M^{d}|q (y_1^d - y_2^d) \theta - a|
\\
&=&  \frac{r}{ \gcd(k, r) }  +   \frac{M^{d} k | r  g^d h \theta - b |}{\gcd(k, r)   }
\\
&\gg&
 \frac{r}{ \gcd(k, r) }  +   \frac{M^{d} N^{d-1} g^{-(d-1)} | r  g^d h \theta - b |}{   \gcd(k, r)   }.
\end{eqnarray*}
Then
$$
E \ll  \sum_{ 1 \leq g < 2N  } E_g,
$$
where
$$
E_g 
=
\sum_{N/g \leq z < 2 N/g}
\sum_{ \substack{  1 \leq h < N/g  \\  \gcd(z + h, z) = 1   }  }
\left(  \frac{\gcd \bigl(  \frac{ (z+ h)^d - z^d }{h}   , r \bigr)   }{ r +   M^{d} N^{d-1} g^{-(d-1)} | r  g^d h \theta - b |} \right)^{ \frac{1}{d(d-1)} }.
$$
By following through the argument to bound~\cite[(3.9)]{KawaWoo}, we obtain
\begin{eqnarray*}
\notag
\sum_{ \substack{ N/g \leq z < 2 N/g  \\  \gcd(h, z) = 1   }  }\gcd \left(  \frac{ (z+ h)^d - z^d }{h}   , r \right)^{ \frac{1}{d(d-1)} }
&\ll& \sum_{  \substack{ s | r \\ 1 \leq s \leq  d 3^{d-1} N^{d-1}/ g^{d-1}   } } s^{ \frac{1}{d(d-1)} }  \sum_{  \substack{  N/ g \leq z < 2 N /g  \\ \gcd (h, z) = 1 \\ s | \frac{ (z+ h)^d - z^d }{h}  }  } 1
\\
\notag
&\ll& \sum_{  \substack{ s | r \\  1 \leq s \leq  d 3^{d-1} N^{d-1}/ g^{d-1}     } } s^{ \frac{1}{d(d-1)} + \varepsilon}  \left(  \frac{N}{ s g }  +  1 \right)
\\
\notag
&\ll&   P^{\varepsilon} N   g^{-1}  + P^{\varepsilon}  N^{\frac{1}{d}}  g^{ - \frac{1}{d}}
\\
\notag
&\ll&  P^{\varepsilon}  N  g^{-1}.
\end{eqnarray*}
Thus, we need to bound
$$
E \ll P^{\ve} \sum_{ 1 \leq g < 2N  } \sum_{  1 \leq h < N/g }
\frac{ N g^{-1}  }{  (r +   M^{d} N^{d-1} g^{-(d-1)} | r  g^d h \theta - b | )^{ \frac{1}{d(d-1)} }  }.
$$

Let $\eta_2 > 0$ be sufficiently small. We have
\begin{eqnarray*}
&& P^{\ve} \sum_{ 1 \leq g < 2N  } \sum_{ \substack{  1 \leq h < N/g \\  r > \eta_2 Q   } }  \frac{   N   g^{-1} }
{(  r +   M^{d} N^{d-1} g^{-(d-1)} | r  g^d h \theta - b | )^{ \frac{1}{d(d-1)}} }
\\
&\ll& P^{\ve} Q^{-  \frac{1}{d(d-1)} }  \sum_{ 1 \leq g < 2N  } \sum_{  1 \leq h < N/g }   N g^{-1}
\\
&\ll& P^{\ve} Q^{-  \frac{1}{d(d-1)} }  N^2
\\
&\ll& P^{2 + \ve} M^{- 2}  Q^{-  \frac{1}{d(d-1)} },
\end{eqnarray*}
which satisfies \eqref{Ebdd}. Similarly, we also have
\begin{eqnarray*}
&& P^{\ve} \sum_{ 1 \leq g < 2N  } \sum_{ \substack{  1 \leq h < N/g \\   M^{d} N^{d-1} g^{-(d-1)} | r  g^d h \theta - b | > \eta_2 Q   } }  \frac{   N g^{-1} }
{(  r +   M^{d} N^{d-1} g^{-(d-1)} | r  g^d h \theta - b | )^{ \frac{1}{d(d-1)}} }
\\
&\ll& P^{2 + \ve} M^{- 2}  Q^{-  \frac{1}{d(d-1)} }.
\end{eqnarray*}
Therefore, from this point on, we may further suppose without loss of generality that
$$
r \leq \eta_2 Q
$$
and
$$
|r g^d h \theta - b| \leq \frac{\eta_2 Q}{M^d N^{d-1} g^{-(d-1)}} \qquad
(1 \leq g < 2N, 1 \leq h < N/g).
$$

Let $1 \leq g < 2N$. By Dirichlet's theorem, there exist $c \in\ZZ$ and $s \in \NN$ such that  $\gcd(c, s)=1$,
\[
1 \leq s \leq M^{\frac{d}{2}} N^{\frac{d}{2}} g^{ - \frac{d}{2} }
\qquad
\textnormal{and}
\qquad
|s g^d \theta - c| \leq  M^{ - \frac{d}{2}} N^{- \frac{d}{2} } g^{\frac{d}{2}}.
\]
Then we have
\begin{eqnarray*}
\left| \frac{b}{r} - \frac{ h c }{s}  \right|
&\leq&  \frac{ | r   g^d h \theta -  b |  }{r} + \frac{ h | s  g^d  \theta -  c  |  }{s}
\\
&\leq& \frac{\eta_2 Q}{ M^d N^{d-1} g^{-(d-1)} r }  + \frac{ h  M^{ - \frac{d}{2}} N^{- \frac{d}{2}} g^{\frac{d}{2}}  }{s}
\\
&\leq& \frac{  \eta_2 s Q M^{-d}   N^{- (d-1) } g^{d-1} + r h  M^{ - \frac{d}{2}} N^{- \frac{d}{2}} g^{\frac{d}{2}}    }{rs}
\\
&<& \frac{ 1 }{r s},
\end{eqnarray*}
where the final inequality follows since $d \geq 3$, $N \leq M$, $Q^4 \leq P \ll M N$  and
\begin{eqnarray*}
\eta_2 s Q M^{-d}   N^{- (d-1) } g^{d-1} + r h  M^{ - \frac{d}{2}} N^{- \frac{d}{2}} g^{\frac{d}{2}}
&\leq& 2 \eta_2 Q M^{- \frac{d}{2}} N^{1 - \frac{d}{2}} g^{-1 + \frac{d}{2}}
\\
&\leq& 2^{d} \eta_2 Q M^{- \frac{d}{2}}
\\
&<& 1.
\end{eqnarray*}
Therefore, it follows that
$$
\frac{b}{r}  =  \frac{ h c }{s}, 
$$
which implies
\begin{eqnarray*}
r +   M^{d} N^{d-1} g^{-(d-1)} | r  g^d h \theta - b |
=   \frac{  s  +  M^{d}  N^{d-1} g^{-(d-1)} h  |s g^d \theta - c| }{\gcd(h, s)}.
\end{eqnarray*}
Thus, we need to bound
\begin{eqnarray}
\label{bddEEE}
E &\ll& P^{\ve} \sum_{ 1 \leq g < 2N  } \sum_{  1 \leq h < N/g }
\frac{ N g^{-1}  \gcd(h, s)^{\frac{1}{d(d-1)}}   }{  (s  +  M^{d}  N^{d-1} g^{-(d-1)} h  |s g^d \theta - c| )^{ \frac{1}{d(d-1)} }  }
\\
\notag
&\ll&
 P^{\ve} \sum_{ 1 \leq g < 2N  } \frac{ N g^{-1}    }{  (s  +  M^{d}  N^{d} g^{-d}  |s g^d \theta - c| )^{ \frac{1}{d(d-1)} }  }
  \sum_{  1 \leq h < N/g }
\frac{ \gcd(h, s)^{\frac{1}{d(d-1)}}   }{  ( N^{-1} g h)^{ \frac{1}{d(d-1)} }  }
\\
\notag
&\ll&
 P^{\ve} \sum_{ 1 \leq g < 2N  } \frac{ N^{1 + \frac{1}{d(d-1)}} g^{-1 - \frac{1}{d(d-1)}}    }{  (s  +  M^{d}  N^{d} g^{-d}  |s g^d \theta - c| )^{ \frac{1}{d(d-1)} }  }
  \sum_{  1 \leq h < N/g }
\frac{ \gcd(h, s)^{\frac{1}{d(d-1)}}   }{  h^{ \frac{1}{d(d-1)} }  }.
\end{eqnarray}

Since
\begin{eqnarray*}
\sum_{  1 \leq h < N/g }
\frac{ \gcd(h, s)^{\frac{1}{d(d-1)}}   }{  h^{ \frac{1}{d(d-1)} }  }
&\leq& \sum_{ \substack{  \ell | s \\ 1 \leq \ell <  N /  g  } }  \sum_{ \substack{ 1 \leq h < N/g \\ \ell | h }  }
\frac{ \ell^{\frac{1}{d(d-1)}}   }{  h^{ \frac{1}{d(d-1)} }  }
\\
&\ll&
\sum_{ \substack{  \ell | s \\ 1 \leq \ell < N /  g  } }    N^{ 1 - \frac{1}{d(d-1)} }  g^{ -1 +  \frac{1}{d(d-1)} }
\\
&\ll&  P^{\ve}  N^{ 1 - \frac{1}{d(d-1)} }  g^{ -1 +  \frac{1}{d(d-1)} },
\end{eqnarray*}
\eqref{bddEEE} becomes
$$
E \ll P^{\ve} \sum_{ 1 \leq g < 2N  } \frac{ N^{2} g^{-2}    }{  (s  +  M^{d}  N^{d} g^{-d}  |s g^d \theta - c| )^{ \frac{1}{d(d-1)} }  }.
$$
Let $v = \gcd(g^d, c)$. Then
$$
s  +  M^{d}  N^{d} g^{-d}  |s g^d \theta - c|
= \frac{v}{g^d} \left( s' + M^d N^d |s' \theta - c'| \right)
\gg g^{-d} \left( s' + P^d |s' \theta - c'| \right),
$$
where $s' = s g^d/v$ and $c' = c/v$.
Since $\theta\notin\mathfrak M(Q)$, we have
\begin{equation*}
Q < s' + P^d |s' \theta - c'|,
\end{equation*}
from which it follows that
\begin{equation*}
\frac{Q}{2} < s' \qquad \textnormal{or} \qquad \frac{Q}{2} <  P^d |s' \theta - c'|.
\end{equation*}
Therefore, it follows that
\begin{eqnarray*}
\notag
E &\ll& P^{\ve} \sum_{ 1 \leq g < 2N  } \left(  N^{2} g^{-2}   Q^{ - \frac{1}{d(d-1)}}
+  \frac{ N^{2 } g^{-2 }   }{  ( g^{-d}  Q  )^{ \frac{1}{d(d-1)} }  } \right)
\\
\notag
&\ll& P^{\ve} N^2  Q^{ - \frac{1}{d(d-1)}} + P^{\ve} N^{2 } Q^{ - \frac{1}{d(d-1)}}
\sum_{ 1 \leq g < 2N  }  g^{-2  + \frac{1}{d-1} }
\\
\notag
&\ll& P^{2 + \ve} M^{-2}   Q^{ - \frac{1}{d(d-1)}},
\end{eqnarray*}
which again satisfies \eqref{Ebdd}. This completes the proof of the lemma.

\end{document}